\documentclass[12,fleqn]{article}
\usepackage{amssymb}
\usepackage{graphicx}
\usepackage{amsmath}
\usepackage{amsthm}

 \newtheorem{theorem}{Theorem}[section]
 \newtheorem{corollary}[theorem]{Corollary}
 \newtheorem{lemma}[theorem]{Lemma}
  \newtheorem{conjecture}[theorem]{Conjecture}
 
 \theoremstyle{definition}
 
 \theoremstyle{remark}
 \newtheorem{remark}[theorem]{Remark}
 \newtheorem{example}{Example}
 \numberwithin{equation}{section}
 
\begin{document}

\title{Difference Covering Arrays and Pseudo-\\Orthogonal Latin Squares}

\author{Fatih Demirkale\\
\small\tt fatih.demirkale@uqconnect.edu.au\\
\small\tt Department of Mathematics,\\
\small\tt  Ko\c{c} University, Sar{\i}yer, 34450, \.{I}stanbul, Turkey\\
\\
Diane M. Donovan  \\
\small\tt dmd@maths.uq.edu.au\\
\small\tt Centre for Discrete Mathematics and Computing,\\
\small\tt University of Queensland, St Lucia 4072 Australia \\
\\
Joanne Hall\\
\small\tt j42.hall@qut.edu.au\\
\small Department of Mathematics\\
\small Queensland University of Technology\\
\small Qld, 4000\\
\small\tt j42.hall@qut.edu.au\\
\\
Abdollah Khodkar\\
\small \tt akhodkar@westga.edu\\
\small \tt Department of Mathematics\\
\small University of West Georgia\\
\small Carrollton, GA 30118, USA\\
\\
Asha Rao\\
\small\tt asha@rmit.edu.au\\
\small School of Mathematical  and Geospacial Sciences\\
\small RMIT University\\
\small Vic 3000, Australia
}
\date{January 19, 2015}
\maketitle

%\subjclass {05B15}

%\keywords{Difference covering arrays, Latin squares, pseudo-orthogonal Latin squares, mutually nearly orthogonal Latin squares}

%\cortext[cor]{corresponding author}

\maketitle

\begin{abstract}
Difference arrays are used in applications such as software testing, authentication codes and data compression.  Pseudo-orthogonal Latin squares are used in experimental designs.   A special class
of pseudo-orthogonal Latin squares are the mutually nearly orthogonal Latin squares (MNOLS) first discussed in 2002, with general constructions given in 2007. In this paper we
develop row complete MNOLS from difference covering arrays. We will use this connection to settle the spectrum question for sets of 3 mutually pseudo-orthogonal Latin squares of even order, for all but the order 146.
\end{abstract}

\section{Introduction}

Difference matrices are a fundamental tool used in the construction of combinatorial objects,   generating a significant body of research that has identified a number of existence constraints. These difference matrices have been used for diverse applications, for instance, in the construction of authentication codes without secrecy \cite{Stinson}, software testing \cite{CDFP1},\cite{CDPP2} and data compression \cite{Korner}.
This diversity of applications, coupled with  existence constraints, has motivated authors to generalise the definition to holey difference matrices,  difference covering arrays and difference packing arrays, to mention just a few.

 In the current paper  we are interested in constructing subclasses of cyclic difference covering arrays and exploiting these  structures to emphasize new connections with other  combinatorial objects, such as  pseudo-orthogonal Latin squares. We use this connection to settle
  the existence spectrum for sets of 3 mutually pseudo-orthogonal Latin squares of even order, in all but one case. We begin with the formal definitions.

A {\em difference matrix} (DM) over an abelian group $(G,+)$ of order $n$ is defined to be an $n\times k$ matrix $Q=[q(i,j)]$ with entries from $G$  such that,
  for all pairs of  columns  $0\leq j,j^\prime\leq k-1$, $j\neq j^\prime$,  the difference set
$$
\Delta_{j,j^\prime}=\{q(i,j)-q(i,j^\prime) \mid 0\leq i\leq n-1\}
$$
contains every element of $G$ equally often, say $\lambda$ times. (See, for instance, \cite{colbourn}, \cite{Ge} and \cite{HSS}.) Note that we label the rows from $0$ to $n-1$ and the columns $0$ to $k-1$. Also  to be consistent with later sections involving  Latin squares and covering arrays  our definition uses the transpose of the matrix given in \cite{colbourn} and \cite{Ge}. Since the addition of a constant vector, over $G$,  to all rows and a constant vector to any column does not alter the set $\Delta_{j,j^\prime}$,  we may assume that one row and one column contain only $0$, the identity element of $G$. More precisely, to simplify later calculations, we will assume that   all entries in the last row and last column of $Q$ are $0$.
A difference matrix will be denoted DM$(n,k;\lambda)$. If $(G,+)$ is the cyclic group we refer to a {\em cyclic difference matrix}.

\begin{theorem}\cite[Thm 17.5, p 411]{colbourn} A DM$(n,k;\lambda)$ does not exist if $k>\lambda n$.
\end{theorem}

In the main, we will use difference matrices with $k= 4$, $\lambda=1$ and where possible we will work with cyclic difference matrices. In Section \ref{sc:spec} we list a number of existence results that will be relevant to the current paper.

A {\em holey difference matrix} (HDM) over an abelian group $(G,+)$ of order $n$ with a subgroup $H$ of order $h$ is defined to be an $(n-h)\times k$ matrix $Q=[q(i,j)]$ with entries from $G$  such that,
  for all pairs of  columns  $0\leq j,j^\prime\leq k-1$, $j\neq j^\prime$,  the difference set
$$
\Delta_{j,j^\prime}=\{q(i,j)-q(i,j^\prime) \mid 0\leq i\leq n-h-1\}
$$
contains every element of $G\setminus H$ equally often, say $\lambda$ times. A holey difference matrix will be denoted HDM$(k,n;h)$, where $|G|=n$ and $|H|=h$. If $G$ is the cyclic group then we refer to a {\em cyclic holey difference matrix}.

\begin{remark}\label{rem}
As before  a constant vector may be added to any column without affecting $\Delta_{j,j^\prime}$ so we may assume that all entries in the last column of $Q$ are equal to $0$. However since $H$ is a subgroup,  $0$ belongs to the hole. Consequently $0$ does not occur in $\Delta_{j,j^\prime}$, and thus there will be no row containing two or more $0$'s. Further since $\Delta_{j,k-1}=G\setminus H$, $0\leq j\leq k-2$, the entries of $H$ do not occur in the first $k-1$ columns of $Q$.
\end{remark}

 A {\em difference covering (packing) array} over an abelian group $(G,+)$ of order $n$ is defined to be an $\eta\times k$ matrix $Q=[q(i,j)]$ with entries from $G$  such that,
  for all pairs of distinct columns  $0\leq j,j^\prime\leq k-1$,  the difference set
$$
\Delta_{j,j^\prime}=\{q(i,j)-q(i,j^\prime) \mid 0\leq i\leq \eta-1\}
$$
contains every element of $G$ at least (at most) once. (See, for instance, \cite{Yin1} and \cite{Yin2}.) A difference covering array will be denoted  DCA$(k,\eta;n)$ and a difference packing array will be denoted DPA$(k,\eta;n)$.
  If $(G,+)$ is the cyclic group, then the difference covering (packing) array is said to be {\em cyclic}.

   Difference covering arrays have been studied in their own right and are related to mutually orthogonal partial Latin squares and transversal coverings, with applications  in information technology, see \cite{KAG} and \cite{SMM}.

As before we may assume that the last row and last column of a DCA$(k,\eta;n)$ contain only 0.

 In the papers \cite{Yin1} and \cite{Yin2}, Yin constructs cyclic DCA$(4,n+1;n)$ for all even integers $n$, with similar results for cyclic difference packing arrays. Yin  documents a number of product constructions for difference covering arrays, some of which will be reviewed in Section \ref{sc:spec} and then adapted to construct difference covering arrays with specific properties; properties that build connections with pseudo-orthogonal Latin squares.

 The additional properties that we  seek are that $0$ (the entry relating to identity element of $G$) occurs at least twice in each column of the  DCA$(k,n+1;n)$ and for pairs of  columns, not including the last column, the repeated difference is not the element $0$.
 Formally we are interested in    DCA$(k,n+1;n)$, $Q=[q(i,j)]$, ($0\leq i\leq n$, $0\leq j\leq k-1$) satisfying the properties:
\begin{itemize}
\item[{\bf P1.}]  the entry $0\in G$ occurs at least twice in each column of $Q$, and
\item[{\bf P2.}] for all pairs of distinct columns $j$ and $j^\prime$, $j\neq k-1\neq j^\prime$, $\Delta_{j,j^\prime}=\{q(i,j)-q(i,j^\prime)\mid 0\leq i\leq n-1\}=G\setminus\{0\},$
\end{itemize}
Note that this last property implies that $\Delta_{j,j^\prime}$ contains a repeated difference that is not $0$.

The following example, of cyclic DCA$(4, 7;6)$ that satisfies  P1 and P2, is taken from \cite{RSS}.
$$
B^T=\left[\begin{array}{ccccccc}
0&1&2&3&4&5&0\\
1&3&5&0&2&4&0\\
3&0&4&1&5&2&0\\
0&0&0&0&0&0&0
\end{array}\right]
$$

In the next lemma we show that if $G$ is the cyclic group over ${\mathbb Z}_n$, then these conditions imply that for all distinct columns $j$ and $j^{\prime}$, $j\neq k-1\neq j^\prime$,
$$
\Delta_{j,j^\prime}=\{0, 1, 2,\dots,n/2,n/2,\dots, n-1\}
$$
with repetition retained.

\begin{lemma}\label{diff} If there exists a  cyclic  DCA$(k,n+1;n)$, $Q=[q(i,j)]$, ($0\leq i\leq n$, $0\leq j\leq k-1$) satisfying Properties P1 and P2,
 then $n$ is even. Further, given $d_0$ such that  $d_0=q(i,j)-q(i,j^\prime)=q(i^\prime,j)-q(i^\prime,j^\prime)$, for  $i\neq i^\prime$  and  $k-1\neq j\neq j^\prime\neq k-1$,    then $d_0 = n/2$.
\end{lemma}

\begin{proof} Let $Q=[q(i,j)]$ ($0\leq i\leq n$, $0\leq j\leq k-1$) represent the difference covering array. The definition requires that ${\mathbb Z}_{n}\subseteq \Delta_{j,j^\prime}$ and since column $k-1$ of $Q$ contains all zeros, Property P1 implies that the remaining columns are permutations of the multi-set $\{0,0, 1, 2,\dots, n-1\}$.

 Let $d_0\in {\mathbb Z}_n\setminus \{0\}$ represent the repeated difference in $\Delta_{j,j^\prime}$. Suppose $n$ is odd and, without loss of generality, that column 0 is in standard form. Then, for all $0< j \leq k-2$, $ \sum_{i=0}^{n-1}q(i,j)=\frac{(n-1)n}{2}$ and
$$
 \sum_{i=0}^{n-1}(i-q(i,j))\equiv \frac{(n-1)n}{2}+d_0 \mod n.
$$
Consequently
$
2d_0=n(2u-n+1),$
or equivalently $n|2d_0$. But since $n$ is odd, this leads to the contradiction, $d_0\in {\mathbb Z}_n$ and $n|d_0$. Thus $n$ is $2p$ for some integer $p$, where $p$ divides $d_0$, implying $d_0=p$.

\end{proof}

The remainder of this paper is organised as follows. In Section \ref{Latinsquare} we will draw the connection between DCA$(k,n+1;n)$ and sets of mutually pseudo-orthogonal Latin squares and for a subclass of squares settle the spectrum question for all but a single order, namely 146. In Section \ref{sc:spec} we review some of the general constructions for difference covering arrays and show that these constructions can be used to construct DCA$(k,n+1;n)$ that satisfy Properties P1 and P2.
In Section \ref{constructions} we give three new constructions for DCA$(4,n+1;n)$'s and consequently new families of
 mutually pseudo-orthogonal Latin squares.

The notation  $[a,b]=\{a,a+1,\dots, b-1,b\}$ refers to  the closed interval of integers from $a$ to $b$.

\section{Pseudo-orthogonal Latin squares and  difference covering arrays}\label{Latinsquare}

In this section we verify that   cyclic difference covering arrays can be used to construct pseudo-orthogonal Latin squares.

A {\em Latin square} of order $n$  is an $n\times n$ array  in which each of the symbols of ${\mathbb Z}_n$ occurs once in every row and once in every column.
 Two Latin squares $A=[a(i,j)]$ and $B=[b(i,j)]$, of order $n$, are said to be {\em orthogonal} if
$$
O=\{(a(i,j),b(i,j))\mid 0\leq i,j\leq n-1\}={\mathbb Z}_n\times {\mathbb Z}_n.
$$
A set of $t$ Latin squares is said to be {\em mutually orthogonal}, $t$-MOLS$(n)$, if they are pairwise orthogonal.
A set of  $t$ {\em idempotent} MOLS$(n)$, denoted  $t$-IMOLS$(n)$, is a set of $t$-MOLS$(n)$ each of which is idempotent; that is, the cell $(i,i)$ contains the entry $i$, for all $0\leq i\leq n-1$.

It is well known that difference matrices can be used to construct sets of mutually orthogonal Latin squares, see for instance  \cite[Lemma 6.12]{HSS}.

While the applications of orthogonal Latin squares are well documented,  there are still many significant existence questions   unanswered. For
instance,   it is known that there is no pair of
MOLS(6), however it  is not known if there exists a set of three MOLS(10),
 or  four MOLS(22), see \cite{colbourn}. The existence of a set of  four MOLS(14) was established by  Todorov \cite{todorov} in 2012, but it is not known if there exists a set of five MOLS(14). Many of the existence results have been obtained using  quasi-difference matrices  or  difference matrices with holes, see \cite{colbourn}.

The importance and applicability of MOLSs combined with these difficult open questions has motivated authors, such as  Raghavarao, Shrikhande and Shrikhande \cite{RSS} and Bate and Boxall \cite{BB},  to slightly vary the  orthogonality condition  to that of pseudo-orthogonal. A pair of Latin squares, $A=[a(i,j)]$ and $B=[b(i,j)]$, of order $n$, is said to be {\em pseudo-orthogonal} if
given
$O=\{(a(i,j),b(i,j))\mid 0\leq i,j\leq n-1\}$,
 for all $a\in {\mathbb Z}_n$
$$
|\{(a,b(i,j))\mid (a,b(i,j))\in O\}|=n-1.
$$
That is, each symbol in $A$ is paired with every symbol in $B$ precisely once, except for one symbol with which it is paired twice and one symbol with which it is not paired at all. A set of $t$ Latin squares, of order $n$, are said to be mutually pseudo-orthogonal if they are pairwise pseudo-orthogonal.

The value and applicability  of pseudo-orthogonal Latin squares  has been established through applications to multi-factor crossover designs in animal husbandry \cite{BB}, and strongly regular graphs \cite{BHS} (though the definition varies here). {\em Mutually nearly orthogonal Latin squares} (MNOLS) are a  special class of pseudo-orthogonal Latin squares, in that the set $O$ does not contain the pair $(a,a)$, for any $a\in{\mathbb Z}_n$. Mutually nearly orthogonal Latin squares (MNOLS) were first discussed in  a paper by   Raghavarao,   Shrikhande and  Shirkhande in 2002 \cite{RSS}.

A natural question to ask is: Can we use difference techniques to construct mutually pseudo-orthogonal Latin squares?
 Raghavarao,   Shrikhande  and  Shirkhande  did precisely this and constructed  mutually pseudo-orthogonal Latin squares from cyclic DCA$(k,n+1;n)$ termed $(k,n)$-difference sets in \cite{RSS}. The  Raghavarao,   Shrikhande  and  Shirkhande result is as follows.

 \begin{theorem}
If there exists a cyclic DCA$(t+1, 2p+1;2p)$, $Q^\prime=[q^\prime(i,j)]$, that satisfies  P1 and P2, then there exists a set of $t$ pseudo-orthogonal Latin squares of order $2p$.
\end{theorem}

\begin{proof} Recall that without loss of generality  we may assume that the last row and column of $Q^\prime$ contain all zeros.  Construct a new matrix  $Q=[q(i,j)]$ by removing the last row and last column from $Q^\prime$ and  define a set of $t$ arrays, $L_s=[l_s(i,j)]$, $0\leq s\leq t-1$, of order $2p$, by
\begin{equation}\label{cyclic}
l_s(i,j)=q(i,s)+j(\mbox{mod }2p),\ 0\leq i,j\leq 2p-1.
\end{equation}
It is easy to see that each column of  $L_s$ is a permutation of ${\mathbb Z}_{2p}$  and so $L_s$ is a Latin square. By Lemma \ref{diff}
$$
\Delta_{j,j^\prime}=\{q^\prime(i,j)-q^\prime(i,j^\prime) \mid 1\leq i\leq 2p\}=({\mathbb Z}_{2p}\setminus \{0\})\cup \{p\}
$$
  implying that when any two Latin squares are superimposed we obtain the set of ordered pairs $(\{{\mathbb Z}_{2p}\times {\mathbb Z}_{2p}\}\setminus\{(x,x)\mid 0\leq x\leq 2p-1\})
\cup \{(x,x+p)\mid 0\leq x\leq 2p-1\}$ with repetition retained.

\end{proof}

If there exists a pair of pseudo-orthogonal Latin squares generated from cyclic difference covering arrays satisfying P1 and P2, then there exists a pair of nearly orthogonal Latin squares. Conversely,
a pair of nearly orthogonal Latin squares are necessarily pseudo-orthogonal Latin squares.
Given this and the strong connection with papers \cite{LvR} and \cite{RSS} we will state all results in terms of mutually nearly orthogonal Latin squares.

 Raghavarao,   Shrikhande  and  Shirkhande established bounds on the maximum number of Latin squares in a set of mutually nearly orthogonal Latin squares. This result provides bounds on $k$ for DCA$(k,n+1;n)$ that satisfy P1 and P2.

 \begin{lemma} Let $p\geq 2$ be a positive integer. If there exists a cyclic DCA$(k+1, 2p+1;2p)$ that satisfies P1 and P2, then $k\leq p+1$. Further if $p$ is even and there exists a DCA$(k, 2p+1;2p)$, then $k<p+1$.
 \end{lemma}

 \begin{proof} If $k>p+1$ then there exists a set of more than $(p+1)$-MNOLS$(2p)$, which contradicts Raghavarao,   Shrikhande  and  Shrikhande result. Similarly for the second statement.
 
 \end{proof}

\subsection{ Some interesting facts}

It is also interesting to note that the MNOLS$(2p)$ constructed from cyclic difference covering arrays are essentially copies of the cyclic group. Consequently, these Latin squares are all {\em bachelor squares}, in that they have no orthogonal mate.

We also note that these sets of Latin squares are row complete. A {\em row complete} Latin square, $L=[l(i,j)]$ is one in which the columns can be reordered in such a way that the set
$\{(l(i,j),l(i,j+1)\mid i \in {\mathbb Z}_n, 0\leq j \leq n-1\}={\mathbb Z}_n\times {\mathbb Z}_n$. So the set of entries obtained by taking pairs of adjacent cells in the same row, for all rows, gives the set of all ordered pairs on ${\mathbb Z}_n$.

Williams \cite{W} verified that the columns of the Latin square corresponding to the Cayley table of the cyclic group can be rearranged to obtain a row complete Latin square. %In particular Williams showed that when the columns of the cyclic Latin square, of order $n=2p$ and denoted $C=[c(i,j)]$ where $c(i,j)=i+j\mod n$ are reordered to
%\begin{align*}
%0\ \;\; 1\ \;\;n-1\ \;\;2\ \;\; n-2\ \;\; 3\ \;n-3\ \; \dots \ \;\; n/2
%\end{align*}
%we obtain a row complete Latin square.

Each of the $k$ MNOLS$(2p)$ constructed from cyclic DCA$(k+1, 2p+1;2p)$ can be obtained by reordering the rows of the Cayley table of the cyclic group, without touching the columns. Hence simultaneously reordering the columns of these nearly orthogonal Latin squares will also produce row complete pseudo-orthogonal Latin squares.

\section{The spectrum for sets of 3 mutually nearly orthogonal   Latin squares}\label{sc:spec}

In 2007,   Li   and  van Rees  \cite{LvR} continued the study of $3$-MNOLS$(n)$ and conjectured that they exist for all even $n\geq 6$. In a partial answer to this question, Li and van Rees proved the existence for small orders and  orders greater than 356,  (see also \cite{PR}).

\begin{theorem}\cite[Thm 4.8]{LvR}
If $2p\geq 358$, then there exists a $3$-MNOLS$(2p)$.
\end{theorem}\emph{}

This work was extended in 2014, when Demirkale, Donovan and Khodkar \cite{DDK} developed further constructions for cyclic DCA$(4, 2p+1;2p)$ proving:

\begin{theorem}\cite{DDK}
There exist $3$-MNOLS$(2p)$, where $2p\equiv $ {\rm 14, 22, 38, 46 mod 48}.
\end{theorem}

The next result lists known values for  cyclic DCA$(4, 2p+1;2p)$ satisfying P1 and P2, with $2p\leq 356$.

\begin{lemma}
There exists  cyclic DCA$(4, 2p+1;2p)$ for $2p= $ {\rm 6, 8,}$\dots$, {\rm 20, 22, 38, 46, 62, 70, 86, 94,
110, 118, 134, 142, 158, 166, 182, 190, 206, 214, 230, 238, 254, 262, 278, 286, 302, 310, 326, 334,
350}.
\end{lemma}

\begin{proof} The existence of orders $6$ and $8$ was given in \cite{RSS} and orders 10, 12, 14, 16, 18, 20 in \cite{LvR}. All the remaining cases were shown to exist in \cite{DDK}.

\end{proof}

van Rees recently summarised these results and indicated that $3$-MNOLS$(2p)$ exist for all orders except possibly those given below.

\begin{lemma} \cite{priv-van-rees}\label{rem-spec} A set of $3$-MNOLS$(2p)$ exists except possibly when $2p=$
{\rm 24, 26, 28, 30, 34, 36, 42, 50, 52, 54, 58, 66, 74, 82, 92, 102, 106, 114, 116, 122, 124,
130, 138, 146, 148, 170, 172, 174, 178}.
\end{lemma}

In this section we will settle the existence question for all even orders except $2p=146$.  For completeness we list all values less than $2p=358$ and, given the connection with row complete Latin squares, where possible we will  use cyclic difference covering arrays to construct the Latin squares. Therefore we begin this section by reviewing relevant results from \cite{Ge}, \cite{Yin1} and  \cite{Yin2}, and adapting these to construct cyclic DCA$(4, 2m+1;2m)$ that satisfy  P1 and P2. Where appropriate we will indicate how these results can be used to  settle the spectrum question.

We begin with the following straight forward result that is analogous to \cite[Lem 2.3]{Ge}.

\begin{lemma}\label{l:insert} Suppose that there exists a HDM$(k,n;h)$ over the group $(G,+)$ with a hole over the subgroup $H$. Further suppose there exists a  DCA$(k,h+1;h)$ over $H$ satisfying  $P1$ and $P2$. Then there exists a DCA$(k,n+1;n)$ over $G$ satisfying  P1 and P2. Further suppose that the HDM$(k,n;h)$ and  DCA$(k,h+1;h)$ are cyclic. Then there exists a cyclic DCA$(k,n+1;n)$ satisfying  P1 and P2.
\end{lemma}

\begin{proof}
In the cyclic case, let $A=[a(i,j)]$ ($0\leq i\leq n-1-h$, $0\leq j\leq k-1$) represent the cyclic HDM$(k,n;h)$ and $B=[b(i,j)]$ $(0\leq i\leq h$, $0\leq j\leq k-1$) represent the  cyclic DCA$(k,h+1;h)$. The definition of cyclic implies  that  $H=\{0,u, 2u, \dots(h-1)u\},$ where $n=uh$ and the proof of Lemma \ref{diff} implies that $h$ is even and the repeated difference in $\Delta_{j,j^\prime}$, $j\neq k-1\neq j^\prime$, of $B$ is $hu/2=n/2$.

Set $Q=[q(i,j)]$ ($0\leq i\leq n$, $0\leq j\leq k$) to be the concatenation of $A$ with $B$ and we obtain  a cyclic DCA$(k,n+1;n)$ that satisfies  P1 and P2.

The non-cyclic case follows similarly.

\end{proof}

Next we give a general product type construction taken from \cite{Ge} and adapt it to construct cyclic difference covering arrays that satisfy P1 and P2.

\begin{lemma}\cite[Lem 2.6]{Ge}\label{l:prod-ge} If both a cyclic HDM$(k,n;h)$ and a cyclic DM$(n^\prime,k;1)$ exist, then so does a cyclic HDM$(k,nn^\prime;hn^\prime)$. In particular, if there exists a cyclic DM$(n,k;1)$ and a cyclic DM$(n^\prime,k;1)$ then there exists cyclic DM$(nn^\prime,k;1)$.
\end{lemma}

The first statement of Lemma \ref{l:prod-ge} coupled with Lemma \ref{l:insert} leads to the following straightforward result.

\begin{corollary}\label{c:prod}
Suppose that there exists a cyclic  HDM$(k,n;h)$, a cyclic DM$(n^\prime,$ $k;1)$ and a cyclic DCA$(k,hn^\prime+1;hn^\prime)$ that satisfies P1 and P2.  Then there exists a cyclic DCA$(k,nn^\prime+1;nn^\prime)$ that satisfies  P1 and P2.
\end{corollary}

The second statement of Lemma \ref{l:prod-ge} can also be adapted.

\begin{lemma}\label{l:prod} Suppose a cyclic DM$(n,k;1)$, a cyclic DM$(n^\prime,k;1)$  and a cyclic DCA$(k,n^\prime+1;n^\prime)$ satisfying P1 and P2  exist. Then there exists a cyclic DCA$(k,$ $nn^\prime+1;nn^\prime)$
that satisfies P1 and P2.
\end{lemma}

\begin{proof} This result can be obtained by taking a hole of size 1 in Corollary \ref{c:prod} or as follows. Let $A=[a(i,j)]$ ($0\leq i\leq n-1,0\leq j\leq k-1$) represent the cyclic DM$(n,k;1)$,
$B=[b(i,j)]$ ($0\leq i\leq n^\prime-1,0\leq j\leq k-1$) represent the cyclic DM$(n^\prime,k;1)$ and $C=[c(i,j)]$ ($0\leq i\leq n^\prime,0\leq j\leq k-1$) represent the cyclic DCA$(k,n^\prime+1;n^\prime)$. Recall that $a(n-1,j)=b(n^\prime-1,j)=c(n^\prime,j)=0$ for all $0\leq j\leq k-1$ and $a(i,k-1)=b(i^\prime,k-1)=c(i^{\prime\prime},k-1)=0$ for all $0\leq i\leq n-1$, $0\leq i^\prime\leq n^\prime-1$, and $0\leq i^{\prime\prime}\leq n^\prime$.

Construct a matrix $Q=[q(i,j)]$ where, for $0\leq i\leq n-1$, $0\leq i^\prime\leq n^\prime-1$, $0\leq i^{\prime\prime}\leq n^\prime$ and $0\leq j\leq k-1$,
\begin{align*}
q(i+i^\prime n,j)&=a(i,j)+b(i^\prime,j)n, \mbox{ when }i\neq n-1\\
q(n-1+ i^{\prime\prime}n,j)&=c(i^{\prime\prime},j)n.
\end{align*}
Then
\begin{align*}
\Delta_{j,j^\prime}&=\{a(i,j)+b(i^\prime,j)n-a(i,j^\prime)-b(i^\prime,j^\prime)n\mid 0\leq i\leq n-2,0\leq i^\prime \leq n^\prime -1\}\\
&\quad \cup  \{c(i^{\prime\prime},j)n-c(i^{\prime\prime},j^\prime)n\mid 0\leq i^{\prime\prime}\leq n^\prime\},\\
&=
{\mathbb Z}_n\setminus\{0,n, 2n,\dots, n(n^\prime-1)\}\cup
\{0,n, 2n,\dots, nn^\prime/2,nn^\prime/2, \dots,\\
& \quad n(n^\prime-1)\}.
\end{align*}
Properties P1 and P2 follow as in the proof of Lemma \ref{l:insert}.

\end{proof}

This result can be generalised to construct non-cyclic difference covering arrays as in \cite{Yin1}.

We now combine these results with various results of \cite{colbourn}, \cite{Ge} and \cite{Yin2} to obtain results for cyclic DCA$(4, 2p+1;2p)$, satisfying  P1 and P2, implying new existence results for row complete $3$-MNOLS$(2p)$, where $2p<358$. In the lists below $*$ indicates that the existence of $3$-MNOLS$(2p)$ was previously unknown.

In doing this we will settle the remaining cases in Lemma \ref{rem-spec}, with the exception of $2p=146$. Here we believe that there exists a DCA$(4, 147;146)$ satisfying P1 and P2 but have been unable to verify it.

\begin{theorem}\cite[Thm 17.6, p 411]{colbourn}\label{prime-cdm} If $n$ is a prime greater than or equal to $k$, then there exists a cyclic DM$(n,k;1)$.
\end{theorem}

\begin{lemma}\cite[Lem 2.5]{Yin2}\label{hdm-2prime}  Let $n\geq 5$ be prime.  Then there exists a cyclic HDM$(4, 2n;2)$.
\end{lemma}

\begin{lemma}
There exists cyclic DCA$(4, 2p+1;2p)$ for $2p=50^*$, {\rm 98}, $170^*$, {\rm 242, 290, 338}.
\end{lemma}

\begin{proof}
Corollary \ref{c:prod} together with Theorem  \ref{prime-cdm} and Lemma \ref{hdm-2prime}  can be used first to construct a cyclic HDM$(4, 2p;h)$ with
$2p(h)=$ {\rm 50(10), 98(14), 170(10), 242(22), 290(10), 338(26)} and then the
required DCA$(4, 2p+1;2p)$.

\end{proof}

Lemmas \ref{spec:90}, \ref{spec:60}, \ref{spec:140} do not document any new existence results, however they do verify that for the given orders
cyclic DCA$(4, 2p+1;2p)$ satisfying P1 and P2 exist.

\begin{lemma}\cite[Thm 2.1]{Yin2}\label{hdm-23}  Let $n$ be an odd positive integer satisfying gcd$(n, 9)\neq 3$.  Then there exists a cyclic HDM$(4, 2n;2)$.
\end{lemma}

\begin{lemma}\label{spec:90}
There exists cyclic DCA$(4, 2p+1;2p)$ for $2p=$ {\rm 90, 126, 198, 234, 306, 342}.
\end{lemma}

\begin{proof}
Corollary \ref{c:prod} together with Theorem  \ref{prime-cdm}  and Lemma \ref{hdm-23}  can be used first  to construct a cyclic HDM$(4, 2p;h)$ with
$2p(h)=$ {\rm 90(10), 126(14), 198(22), 234(26), 306(34), 342(38)} and then the
required DCA$(4, 2p+1;2p)$.
	
\end{proof}

\begin{theorem}\cite[Thm 2.3]{Yin2}\label{hdm-4prime}  Let $n\geq 4$ and $n=2^\alpha 3^\beta p_1^{\alpha_1}\dots p_t^{\alpha_t},$ where $(\alpha,\beta)\neq (1,0)$, $\alpha_i \geq 0$ and the prime factors $p_i\geq 5$ for $1 \leq i \leq t$.  Then there exists a cyclic HDM$(4, 2n;2)$.
\end{theorem}

\begin{lemma}\label{spec:60}
There exists cyclic DCA$(4, 2p+1;2p)$ for $2p=$ {\rm 60, 80, 84, 100, 112, 120, 132, 156, 160, 168,
176, 180, 204, 208, 224, 228, 240, 252, 264, 272, 276, 300, 304,
312, 320, 336, 352}.
\end{lemma}

\begin{proof}
Corollary \ref{c:prod} together with Theorem  \ref{prime-cdm}  and Theorem \ref{hdm-4prime}  can be used first to construct a cyclic HDM$(4, 2p;h)$ with
$2p(h)=$ 60(10), 80(10), 84(14), 100(10), 112(14), 120(10), 132(22),
156(26), 160(10), 168(14),
176(22), 180(10),
204(34), 208(26), 224(14), 228(38), 240(10), 252(14),
264(22), 272(34), 276(46),
300(10), 304(38),
312(26), 320(10), 336(14), 352(22)  and then the
required \\ DCA$(4, 2p+1;2p)$.
	
\end{proof}

\begin{theorem}\cite[Thm 2.4]{Yin2}\label{hdm-4hole}  Let  $n= p_1^{\alpha_1}\dots p_t^{\alpha_t}$,
 where $\alpha_i \geq 0$ and the prime factors $p_i\geq 5$ for $1 \leq i \leq t$.  Then there exists a cyclic HDM$(4, 4n;4)$.
\end{theorem}

\begin{lemma}\label{spec:140}
There exists cyclic DCA$(4, 2p+1;2p)$ for $2p=$ {\rm 140, 196, 220, 260, 308, 340}.
\end{lemma}

\begin{proof}
Corollary \ref{c:prod} together with Theorem  \ref{prime-cdm} and Theorem \ref{hdm-4hole}  can be used  first to construct a cyclic HDM$(4, 2p;h)$ with
$2p(h)=$ {\rm 140(28), 196(28), 220(20), 260(20), 308(28), 340(20)} and then the
required DCA$(4, 2p+1;2p)$.

\end{proof}

\begin{theorem}\cite[Thm 3.10]{Ge}\label{3-cdm} A cyclic DM$(3^i, 5;1)$ exists for all $i\geq 3$.
\end{theorem}

\begin{lemma}
There exists cyclic DCA$(4, 2p+1;2p)$ for $2p=$ {\rm 216, 270, 324}.
\end{lemma}

\begin{proof}
Corollary \ref{c:prod} together with Theorem  \ref{3-cdm}  and Theorem \ref{hdm-4prime}  can be used to first construct a cyclic HDM$(4, 2p;h)$ with
$2p(h)=$ 216(54), 270(54), 324(54) and then the
required DCA$(4, 2p+1;2p)$.

\end{proof}

The next result from Yin's paper \cite{Yin2} is interesting in that it allows us to construct difference covering arrays and so nearly orthogonal Latin squares
of order $6n$, and gives many values that were previously unresolved (the obstruction was the non-existence of MNOLS of order $6$).

\begin{theorem} \cite[Thm 2.2]{Yin2}\label{hdm-6} Let $n$ be a positive integer of the form $p_1^{\alpha_1}p_2^{\alpha_2}\dots p_t^{\alpha_t}$, where $\alpha_i \geq 0$ and
the prime factors $p_i\geq 5$ for $1 \leq i \leq t$. Then there exists a cyclic HDM$(4, 6n;6)$.
\end{theorem}

\begin{corollary}\label{ls:sixes}  Let $n$ be an integer of the form $p_1^{\alpha_1}p_2^{\alpha_2}\dots p_t^{\alpha_t}$, where $\alpha_i \geq 0$ and the prime factors $p_i\geq 5$ for $1 \leq i \leq t$. Then there exists a cyclic DCA$(4, 6n+1;6n)$ satisfying  P1 and P2. Consequently there exists  cyclic DCA$(4, 2p+1;2p)$ for $2p=30^*, 42^*, 66^*, 78, 102^*, 114^*, 138^*, 150, 174^*,\\
186, 210, 222, 246, 258, 282,
294, 318, 330, 354$.
\end{corollary}

The following result can be verified using direct constructions given in the later Section \ref{constructions}.

 \begin{lemma}
There exists cyclic DCA$(4, 2p+1;2p)$ for $2p=26^*$, {\rm 266}, $2p=$ {\rm  40, 56, 88, 104, 136, 152,
184, 200, 232, 248, 280, 296, 328, 344} and $2p=34^*$, $58^*$, $82^*$, $106^*$, $130^*$, {\rm 154},
$178^*$, {\rm 202, 226, 250, 274,
298, 322, 346}.
\end{lemma}

\begin{proof}
Corollary \ref{cor-2m}, given in Section \ref{constructions}, verifies the existence of cyclic DCA$(4, 2p+1;2p)$ with
$2p=26^*, 266.$

Corollary \ref{cor-4m+1}, given in Section \ref{constructions}, verifies the existence of cyclic DCA$(4, 2p+1;2p)$ with
$2p=$ 40, 56, 88, 104, 136, 152, 184, 200, 232, 248, 280, 296, 328, 344.

Theorem  \ref{thm-6mu+4}, given in Section \ref{constructions}, verifies the existence of cyclic DCA$(4, 2p+1;2p)$ with
$2p=34^*, 58^*, 82^*, 106^*, 130^*$, 154, $178^*$, 202, 226, 250, 274, 298, 322, 346.

\end{proof}

\begin{lemma}
There exists cyclic DCA$(4, 2p+1;2p)$ for $2p=24^*, 28^*$, {\rm 32}, $36^*$, {\rm 44, 48}, $52^*$, $54^*$.
\end{lemma}

\begin{proof}
These results have been verified by computer searches. The first column of the DCA$(4, 2p+1, 2p)$ is given by
$[0, 1, 2, \ldots, 2p-1]$, the second column by $[1, 3, \ldots, 2p-1, 2, 4, \ldots, 2p-2]$ and the third column
by

\noindent $2p=24$: [2, 0, 3, 1, 14, 21, 20, 19, 23, 15, 6, 18, 16, 10, 17, 8, 11,
   22, 5, 13, 4, 9, 7, 12]

\noindent $2p=28$: [2, 0, 3, 1, 11, 16, 22, 25, 20, 23, 4, 8, 21, 5, 18, 10, 19,
    13, 24, 27, 7, 26, 15, 9, 6, 14, 17, 12]

\noindent $2p=32$: [2, 0, 3, 6, 1, 13, 22, 30, 21, 25, 28, 26, 7, 5, 23, 20, 12,
    10, 24, 17, 31, 15, 29, 27, 11, 14, 4, 9, 8, 19, 18, 16]

\noindent $2p=36$: [5, 35, 13, 20, 11, 9, 1, 31, 10, 2, 30, 33, 4, 34, 32, 25, 28, 16,
    27, 22, 3, 29, 19, 24, 18, 15, 6, 23, 17, 7, 0, 8, 14, 12, 21, 26]

\noindent $2p=44$: [39, 13, 26, 21, 35, 3, 17, 16, 40, 28, 38, 25, 6, 10, 34, 5, 18, 30,
43, 15, 19, 36, 7, 24, 32, 14, 4, 0, 31, 12, 2, 9, 23, 37, 11, 42,
41, 29, 20, 1, 33, 27, 8, 22]

\noindent $2p=48$: [5, 41, 23, 40,  1, 39, 34, 25, 28,  8,  4,  9, 21, 30, 43, 18, 12,  2, 42, 45,
     32, 37, 33,  0, 26, 15, 13, 22, 10, 35, 44,  7, 36, 16, 27, 19, 46, 38,
     3, 47, 31, 29, 17, 14, 11, 24, 20,  6]

\noindent $2p=52$: [18, 12, 50, 37, 16, 6, 45, 4, 31, 34, 47, 21, 29, 2, 5, 22, 38, 3, 39,
    27, 0, 15, 51, 7, 28, 24, 42, 40, 48, 32, 9, 26, 20, 11, 1, 41, 19,
    35, 43, 13, 49, 33, 14, 17, 46, 8, 36, 23, 10, 30, 25, 44]

\noindent $2p=54$: [6, 5, 31, 27, 20, 38, 19, 4, 30, 51, 3, 52, 49, 14, 48, 23, 41, 12, 25,
    0, 32, 40, 21, 50, 9, 45, 16, 1, 46, 11, 28, 42, 47, 35, 39, 2, 22, 13,
    34, 33, 24, 44, 15, 53, 7, 17, 37, 36, 26, 18, 10, 43, 29, 8].

\end{proof}

For the remaining values
64, 68, 72, 74, 76, 92, 96, 108, 116, 122, 124, 128, 144, 146, 148, 162, 164, 172, 188, 192, 194, 212, 218, 236, 244, 256, 268, 284, 288, 292, 314, 316, 332, 348, 356
we were unable to construct cyclic difference covering arrays however for completeness and to answer questions about the spectrum we give full details verifying existence.
 It should be noted that it is possible to construct DCA$(4, 2p+1;2p)$ satisfying P1 and P2 for some of these orders however our construction does not give cyclic difference covering arrays and so the details have been omitted here.

For the Li and van Rees conjecture \cite[Conjecture 5.1]{LvR} we require two more results from their paper.
The second result uses group divisible designs:  A ${\mathcal K}$-{\em group divisible design} of {\em type} $g_1^{a_1}g_2^{a_2}\dots g_s^{a_s}$   is a partition ${
\mathcal G}$ of a finite set ${\mathcal V}$, of cardinality $v=\sum_{i=1}^s a_ig_i$, into $a_i$ {\em groups} of size $g_i$, $1\leq i\leq s$,
together with a family of subsets ({\em blocks}) ${\mathcal B}$  of ${\mathcal V}$ such that: 1) if $B \in {\mathcal  B}$,
then $|B| \in {\mathcal  K}$, 2) every pair of distinct elements of ${\mathcal V}$ occurs in $1$ block of ${\mathcal B}$ or $1$ group of ${\mathcal G}$ but not both, and  3) $|{\mathcal G}| > 1$.

\begin{theorem}\label{vanrees1}\cite[Thm 4.1]{LvR} Suppose there exists $k$-MNOLS$(2p)$, $k$-MOLS$(2p)$, and $k$-MOLS$(n)$. Then there exists $k$-MNOLS$(2pn)$.
\end{theorem}

\begin{theorem}\label{vanrees2}\cite[Thm 4.5]{LvR}\label{LvR} Suppose there exists a ${\mathcal K}$-GDD of type $g_1^{a_1}\dots g_s^{a_s}$. Further suppose that for any group  size $g_i$ there exists a $s$-MNOLS$(g_i)$ and for any block  size $k\in {\mathcal  K}$ there exists a $s$-IMOLS$(k)$.  Then there are $s$-MNOLS$(\sum_{i=1}^s a_ig_i)$.
\end{theorem}

\begin{lemma}
There exists  $3$-NMOLS$(2p)$ for $2p =$ {\rm 76}, $92^*$, {\rm 96, 108}, $116^*$, $124^*,$
{\rm 128, 144}, $148^*$, {\rm 164},
$172^*$, {\rm 188, 192, 212, 236, 244, 256, 268, 284, 288, 292, 316, 332, 348}  and {\rm 356}.
\end{lemma}

\begin{proof}
%44(8^5 4^1), 48(8^5 8^1),
The $3$-MNOLS$(2p)$, $2p= 76(12^5 16^1), 92(20^4 12), 96(16^5 16^1), 108(20^5 8^1),$ \\ $ 116(20^5, 16),
124(20^5 24^1), 128(24^5 8^1), 144(24^5 24^1), 148(24^5 28), 164(28^5 24^1),$\\ $ 172(32^5 12), 188(32^5 28^1),
192(32^5 36^1), 212(36^5 32^1), 236(40^5 36^1), 244(40^5 44^1),$ \\ $256(44^5 36^1), 268(44^5 48^1), 284(48^5 44^1),
288(48^5 48^1), 292(48^5 52^1), 316(52^5 56^1),$ \\$ 332(56^5 52^1), 348(56^5 68^1)$ and $356(60^5 56^1)$ can be constructed applying $5$-GDD, that exist by \cite[Thm 4.17, p 258]{colbourn}, in Theorem \ref{vanrees2}. Here the bracketed information gives the type of the GDD.
\end{proof}

%Finally for completeness we give the GDD's used for all remaining cases.

%\begin{lemma}There exists  $3$-NMOLS$(2p)$ for all remaining $2p\leq 356$ except possibly $2p=146$.
%\end{lemma}

\begin{lemma}There exists  $3$-NMOLS$(2p)$ for $2p =$ {\rm 64, 68, 72}, $74^*$, {\rm 122, 162, 194, 218}
and {\rm 314}.
\end{lemma}

\begin{proof}
$3$-NMOLS$(2p)$ for $2p =$ 64, 68, 72, $74^*$, 122, 162, 194 and 218 can be constructed by applying $8$-GDD$(8^8)$, $\{7, 8, 9\}$-GDD$(8^7 6^2)$, $9$-GDD$(8^9)$, $\{7, 8, 9\}$-GDD$(8^7 6^3)$, $\{7, 8\}$-GDD$(16^7 10^1)$, $\{11, 12, 13\}$-GDD$(12^{12} 10^1 8^1)$,\\ $\{11, 12, 13\}$-GDD$(16^{11} 10^1 8^1)$,
%$\{11, 12, 13, 14\}$-GDD$(16^{11} 14^3)$ and \\
$\{13, 14\}$-GDD$(16^{13} 10^1)$ and
$\{8, 9, 10,$ $11\}$-\\GDD$(32^8 14^2 10^1)$, that exist by finite field constructions, in Theorem \ref{vanrees2}, respectively.
\end{proof}

All the results of this section combine to the following theorem.

\begin{theorem} There exists a set of $3$-MNOLS$(2p)$ for each positive integers $p\geq 3$, except possibly $p=73$.
\end{theorem}

\begin{conjecture} There exists DCA$(4, 2p+1;2p)$ satisfying P1 and P2 of all positive integers $p\geq 3$.
\end{conjecture}

\section{Construction of difference covering arrays DCA$(4, 2m+1;2m)$}\label{constructions}

This section is devoted to giving new constructions for families of cyclic DCA$(4,$ $2m+1;2m)$ when $m = 2k + 1$, $ m = 8k + 4$ and $ m = 3k + 2$, respectively. In each of these cases the difference covering arrays satisfy P1 and P2 and so they can be used to construct MNOLS$(2m)$.

When using  a cyclic DCA$(4, 2m+1;2m)$ to construct nearly orthogonal Latin squares we strip off the last row and last column of zeros. Thus to  reduce the complexity of the notation and to avoid confusion, we will assume that we are constructing a $2m\times 3$ array  $Q=[q(i,j)]$ that satisfies:
\begin{itemize}
\item each column is a permutation of $Z_{2m}$ and
\item $\Delta_{j,j^\prime}=\{q(i,j)-q(i,j^\prime)|\mid 0\leq i\leq 2m-1\}=\{1, 2,\dots, m,m,\dots, 2m-1\}$, with repetition retained.
\end{itemize}
Also  we will use the following notation:  $q(a,0)=a$ (or $q(\alpha,0)=a(\alpha)$),  $q(a, 1)=b(a)$ (or $q(\alpha, 1)=b(\alpha)$),   and
 $q(a, 2)=c(a)$ (or $q(\alpha, 2)=c(\alpha)$).

The following lemmas document some well known results, stated without proof, which will be used extensively in the proof of subsequent results.

\begin{lemma}\label{basic0}
For all integers $x,y,z$, ${\rm gcd}(x+yz,z)=gcd(x,z).$
\end{lemma}

\begin{lemma}\label{basic}
Let $g$ and $p$ be positive integers and  $h$ a non-negative integer. Working modulo $2p$, if $gcd(g, 2p)=1$ then
\begin{align*}
\{gx+h\mid 0\leq x\leq 2p-1\}&={\mathbb Z}_{2p},
\end{align*}
or if $gcd(g, 2p)=r$ and $h\equiv s\mod r$ then
\begin{align*}
\{gx+h\mid 0\leq x\leq 2p/r-1\}
&=\{rx+s\mid 0\leq x\leq 2p/r-1\}.
\end{align*}
\end{lemma}

 \subsection{Construction for general families DCA$(4, 2m+1;2m)$ for some odd $m$}

In this subsection we give a general construction for a difference covering array DCA$(4, 2m+1;2m)$, for $m$ odd. The proof that such a difference covering array exists uses the results presented in the following lemma. Note that in this section unless otherwise stated all arithmetic is modulo $2m$.
In particular, for $i \not\equiv 2\; \rm{mod}\; 3$ a non-negative integer, and $k = 2i^2 + 7i + 6$, we present an infinite family of DCA$(4, 2m+1;2m)$ for $m = 2k + 1$.

\begin{lemma}\label{lem:gcd}
Let $f$ and $m$ be integers such that
$gcd(f, 2m)=2$,
 $gcd(f+2, 2m)=2$, and
$f^2+f+1\equiv m\mod 2m$.
Then
\begin{align}
{\rm gcd}(f,m)&=1,\label{gcd f}\\
{\rm gcd}(f+1,m)&=1,\label{gcd f+1}\\
{\rm gcd}(f-1,m)&=1,\label{gcd f-1}\\
{\rm gcd}(2f+1,m)&=1\label{gcd 2f+1}\\
mf&\equiv 0\mod 2m.\label{mf=0}
\end{align}
\end{lemma}
\begin{proof}

Eq \ref{gcd f}: \quad Note that since $f$ is even, $f^2+f+1$ is odd implying $m$ is odd and hence ${\rm gcd}(f,m)=1$.

Eq \ref{gcd f+1}: \quad Since $f+1  \equiv -f^2\mod m$ and ${\rm gcd}(f,m)=1$ we have
$1={\rm gcd}(f,m)={\rm gcd}(f^2,m)={\rm gcd}(f+1,m).$

Eq \ref{gcd f-1}: \quad Since
$
f-1 = (-f^2-f-1)+f^2+2f  \equiv f(f+2) \mod m$, ${\rm gcd}(f,m)=1$ and ${\rm gcd}(f+2,m)=1$, we have
$1= {\rm gcd}(f(f+2),m)={\rm gcd}(f-1,m).$

Eq \ref{gcd 2f+1}: \quad Since $2f+1 \equiv -f(f-1)\mod m$, ${\rm gcd}(f,m)=1$ and ${\rm gcd}(f-1,m)=1$,
$1={\rm gcd}(f(f-1),m)={\rm gcd}(2f+1,m).$

Eq \ref{mf=0}:  This follows from the fact that $f$ is even.
\end{proof}

 For a suitable choice of $f$, we divide the domain of $a$ into the subintervals $[0,m+f]$, $[m+f+1,m-1]$, $[m,m-f-1]$ and $[m-f, 2m-1]$ where all endpoints  are included.

\begin{example}
To aid understanding we begin with an example where
$m=13$ and $f=16$ and give the transpose of the difference covering array DCA$(4, 27;26)$. The key to understanding the proof is to recognise that within the
subintervals $I_1 =[0,\dots, 3], I_2 = [4, \dots, 12], I_3=[13, \dots, 22]$ and $I_4 = [23, \dots, 25]$, the value of $a$, $b(a)$ and $c(a)$ increases by a constant ``jump'', respectively $1$, $f=16$ and $-(f+1)=9$. This implies that the differences will also increase by a constant. By carefully choosing the start value on each subinterval it is possible to obtain the required values in ${\mathbb Z_{2m}}$. The value $m=13$ is boldfaced in the differences.

{\scriptsize
$$
\begin{array}{r|cccc|ccccccccc|cccccccccc|ccc}
&\multicolumn{4}{|c|}{I_1}&\multicolumn{9}{|c|}{I_2}&\multicolumn{10}{|c|}{I_3}&\multicolumn{3}{|c}{I_4}\\
\hline
 a  & 0 & 1 & 2 &
  3 & 4 & 5 &
  6 & 7 & 8 &
  9 & 10 & 11 &
  12 & 13 & 14 &
  15 & 16 & 17 &
  18 & 19 & 20 &
  21 & 22 & 23 &
  24 & 25 \\
 b(a)  & 13 & 3 &
  19 & 9 & 25 &
  15 & 5 & 21 &
  11 & 1 & 17 &
  7 & 23 & 2 &
  18 & 8 & 24 &
  14 & 4 & 20 &
  10 & 0 & 16 &
  6 & 22 & 12 \\
 c(a)  & 15 & 24 & 7 &
  16 & 12 & 21 &
  4 & 13 & 22 &
  5 & 14 & 23 &
  6 & 0 & 9 &
  18 & 1 & 10 &
  19 & 2 & 11 &
  20 & 3 & 25 &
  8 & 17 \\
 b(a)-a  & {\bf 13} & 2 & 17 & 6 &
  21 & 10 & 25 &
  14 & 3 & 18 &
  7 & 22 & 11 &
  15 & 4 & 19 &
  8 & 23 & 12 &
  1 & 16 & 5 &
  20 & 9 & 24 &
  {\bf 13}  \\
 c(a)-a  & 15 & 23 & 5 &
  {\bf 13} & 8 & 16 & 24 &
  6 & 14 & 22 &
  4 & 12 & 20 & {\bf 13} & 21 & 3 &
  11 & 19 & 1 &
  9 & 17 & 25 &
  7 & 2 & 10 &
  18     \\
 c(a)-b(a) & 2 & 21 & 14 &
  7 & {\bf 13} & 6 & 25 & 18 &
  11 & 4 & 23 &
  16 & 9 & 24 &
  17 & 10 & 3 &
  22 & 15 & 8 &
  1 & 20 & {\bf 13} & 19 & 12 &
  5 \\
\end{array}
$$
}
\end{example}

\begin{theorem} \label{QD(2m, 3)}
Let $f$ and $m$ be natural numbers such that
${\rm gcd}(f, 2m)=2$, ${\rm gcd}(f+2, 2m)=2$,
$f^2+f+1\equiv m\mod 2m$, and
$m+3\leq f\leq 2m-4.$
Then a cyclic DCA$(4, 2m+1;2m)$ satisfying P1 and P2 exists.
\end{theorem}

\begin{proof}
The proof is by construction with  the values of DCA$(4, 2m+1;2m)$ as given in Figure \ref{values-2m1}, with the 3 columns of $Q=[q(i,j)]$ given by $q(a,0)=a$, $q(a, 1)=b(a)$, $q(a, 2)=c(a)$. We will show that $Q$ has the required properties.

\begin{figure}

\begin{center}
{\scriptsize
\begin{tabular}{|c|c|c|c|c|}
\hline
Intervals& $I_1$&$I_2$&$I_3$&$I_4$\\
\hline
 $q(a,0)=a$& $[0,m+f]$&$[m+f+1,m-1]$&$[m,m-f-1]$&$[m-f, 2m-1]$\\\hline
$q(a, 1)=b(a)$  & $af+m $&$af+m $&$(a+1)f$&$(a+1)f$\\
  & &&$+m-1$&$+m-1 $\\\hline
$q(a, 2)=c(a)$  &$-(a-1)(f+1)  $&$-(a-1)(f+1)$&$-a(f+1) $&$-a(f+1)  $\\
 &$-2$&$+m-2$&$+m $&\\\hline \hline
$b(a)-a$&  $a(f-1) $&$a(f-1) $&$(a+1)(f-1)       $&$(a+1)(f-1)         $\\
&$+m $&$+m $&$+m       $&$+m         $\\
\hline
%& &&&\\
$c(a)-a$&$-(a-1)(f+2) $&$-(a-1)(f+2)$&$-a(f+2) $&$-a(f+2)$\\
&$-3$&$+m-3$&$+m $&$  $\\
\hline
$c(a)-b(a)$&$-a(2f+1)$&$-a(2f+1)$&$-a(2f+1) $&$-a(2f+1)  $\\
&$+f-1+m$&$+f-1$&$-(f-1) $&$-(f-1)+m  $\\
\hline
\end{tabular}
}
\end{center}
\caption{Entries are elements of ${\mathbb Z}_{2m}$, where $q(a,0)=a,$  $q(a, 1)=b(a)$ and $q(a, 2)=c(a)$ in the array $Q=[q(i,j)]$. Rows 5 to 7 give the differences.}\label{values-2m1}
\end{figure}

If $m+3\leq f\leq 2m-4$ and $f$ is even then, working modulo $2m$, $3\leq m+f\leq m-4$ and so   the intervals $[0,m+f]$ and $[m+f+1,m-1]$ are non-empty. Further, $m+4\leq m-f \leq 2m-3$ and so the intervals $[m,m-f-1]$ and $[m-f, 2m-1]$ are non-empty.

Given that  $f$ is even and $m$ is odd,  by Lemma \ref{basic}
$$ af+m\equiv 1\mod 2 \Longrightarrow \{b(a)\mid a\in I_1\cup I_2\}=\{2g+1\mid 0\leq g \leq m-1\},$$
$$(a+1)f+m-1\equiv 0\mod 2 \Longrightarrow \{b(a)\mid a\in I_3\cup I_4\}=\{2g\mid 0\leq g \leq m-1\},$$
and so $\{b(a)\mid 0\leq a\leq 2m-1\}=[2m]$.

On each of the subintervals $c(a)$ takes the form $ag+h$, where $g=-(f+1)$,  so the ``jump'' size is $-(f+1)$ and
$$\begin{array}{rcl}
c(m)&=&0,\\
c(m+f+1)-c(m-f-1)&=&%-(m+f+1-1)(f+1)+m-2+(m-f-1)(f+1)-m)\\
-m-f(f+1)+m-2+m\\
&&-f(f+1)-(f+1)-m\\
&=&-2f^2-2f-2-(f+1)=-(f+1),\\
c(0)-c(m-1)&=&%&f+1-2-(-(m-1-1)(f+1)+m-2)\\
%&=&f+1 +m-2(f+1)-m=
-(f+1),\\
c(m-f)-c(m+f)&=&%&-(m-f)(f+1)-(-(m+f-1)(f+1)-2)\\
%&=& -m+f(f+1) +m+f(f+1)-(f+1)+2\\
%&=&2f^2+2f+2-(f+1)=
-(f+1),\\
c(m)-c(2m-1)&=&%&0-(-(2m-1)(f+1))=
-(f+1).
\end{array}$$
Thus by reordering the subintervals as  $I_3,I_2,I_1,I_4$, and noting for instance, $c(m-f-1)-(f+1)=c(m+f+1)$, we get $\{c(a)\mid 0\leq a\leq 2m-1\}=[2m]$.

For $b(a)-a$, $c(a)-a$ and $c(a)-b(a)$ we are required to show that for $a\in[2m]$ the differences cover the multiset $\{1, 2,\dots,m-1,m,m,m+1,\dots, 2m-1\}=([2m]\setminus \{0\})\cup\{m\}$.

For $b(a)-a$, the subintervals are taken in natural order $I_1,I_2,I_3,I_4$. Starting at $a=0$ and finishing at $a=2m-1$, we have $b(0)-0=m=(2m-1+1)(f-1)+m= b(2m-1)-(2m-1)$, so the difference $m$ occurs twice.
Further, the ${\rm gcd}(f-1, 2m)=1$ implies  that $a(f-1)+m$,  $0\leq a\leq m-1$, are all distinct, as are $(a+1)(f-1)+m$, $m\leq a\leq 2m-1$ and
 \begin{align*}
 b(m)-m&=(m+1)(f-1)+m=f-1,\\
 b(m-1)-(m-1)&=(m-1)(f-1)+m=-(f-1).
\end{align*}
Thus there is a jump of $-2(f-1)$ between $a=m-1$ and $a=m$ and the difference 0 is omitted, implying  $\{b(a)-a\mid0\leq a\leq 2m-1\}=([2m]\setminus \{0\})\cup\{m\}$.

For $c(a)-a$, since $f+2$ is even and ${\rm gcd}(f+2,m)=1$,  these values are all distinct on each of the subintervals, $|I_3\cup I_1|=m+1$, $|I_4\cup I_2|=m-1$, and
\begin{align*}\hspace*{-1cm}
c(m)-m&=-m(f+2)+m=m,\\
c(m+f)-(m+f)&=-(m+f-1)(f+2)-3=m,\\
%&=-f(f+2)+(f+2)-3=m\\
(c(0)-0)-(c(m-f-1)-(m-f-1))&=%(f+2)-3-(-(m-f-1)(f+2)+m)=-(f+2),\\
m-f(f+2)-3=-(f+2),\\
c(2m-1)-(2m-1)&=-(2m-1)(f+2)=f+2,\\
c(m+f+1)-(m+f+1)&=%-(m+f+1-1)(f+2)+m-3
-f^2-2f+m-3
=-(f+2).
%c(2m-1)-(2m-1)-(c(m+f+1)-(m+f-1))&=2(f+2),
\end{align*}
Thus $-(a-1)(f+2)-3,-a(f+2)+m\equiv 1\mod 2$, hence,
$$\{c(a)-a\mid a\in I_3\cup I_1\}=\{2g+1\mid 0\leq g\leq m-1\}\cup\{m\}.$$
In addition, $-a(f+2)+m-1,-a(f+2)\equiv 0\mod 2$ implies that
$$\{c(a)-a\mid a\in I_4\cup I_2\}=\{2g\mid 1\leq g\leq m-1\},$$
giving $\{c(a)-a\mid0\leq a\leq 2m-1\}=([2m]\setminus \{0\})\cup\{m\}$.

For $c(a)-b(a)$, since  ${\rm gcd}(2f+1, 2m)=1$, these values are all distinct on the subintervals, $c(m+f+1)-b(m+f+1)=-2f^2-2f-2+m=m=m+2f^2+2f+2=
c(m-f-1)-b(m-f-1)$, and
\begin{align*}
c(0)-b(0)-(c(m-1)-b(m-1))&%=(f-1+m)-(-(m-1)(2f+1)+f-1)\\
=-(2f+1),\\
c(m+f)-b(m+f)&=%-(m+f)(2f+1)+f-1+m\\
%&=&-2f^2-1=-2(m-f-1)-1
2f+1,\\
c(m-f)-b(m-f)&%=&-(m-f)(2f+1)-(f-1)+m\\
%&=&2f^2+1=2(m-f-1)+1
=-(2f+1).
%(c(m-f)-b(m-f))-(c(m+f)+b(m-f))&=-2(2f+1)\\
%(c(m)-b(m))-(c(2m-1)-b(2m-1))&%=&-m(2f+1)-(f-1)-(-(2m-1)(2f+1)-(f-1)+m)\\
%&=&-m - f+1-2f-1+f-1-m
%= -(2f+1).
\end{align*}
Thus when the subintervals are reordered to $I_2,I_1,I_4,I_3$ we may verify that
 $\{c(a)-b(a)\mid 0\leq a\leq 2m-1\}=([2m]\setminus \{0\})\cup\{m\}$. Note that the values of $c(a)-b(a)$ start and finish on $m$ and the value $0$ is omitted between $a=m+f$ and $a=m-f$.
\end{proof}

\begin{corollary}\label{cor-2m}
Let $i \not\equiv 2\; \rm{mod}\; 3$ be a non-negative integer, $k = 2i^2 + 7i + 6$ and $m = 2k + 1$.
 Then there exists an infinite family of  cyclic DCA$(4, 2m+1;2m)$'s satisfying P1 and P2.
 \end{corollary}

\begin{proof}
Taking $f =  m + 3 + 2i$, then $f$ is even.
In addition $m + 3 \leq f \leq 2m -4 $, since $ i \geq 0$ and $2m - 4 = (m + 3) + (m - 7) = (m + 3) + ( 4i^2 + 14 i + 6 ) \geq m + 3 + 2i = f$. Now
\begin{align*}
f^2 + f + 1 %& = (m + 3 + 2i)^2 + (m + 3 + 2i) + 1\\
                   & \equiv 4i^2 + 14i + 13 \; \rm{mod}\; 2m\\
                & = 2(2i^2 + 7i + 6 ) + 1 = m \; \rm{mod}\; 2m.
\end{align*}
Further, applying Lemma \ref{basic0} repeatedly,
\begin{align*}
\gcd(f, 2m) &= 2(\gcd(2i^2+ 8i + 8, (2i^2+ 8i + 8) + 2i^2 + 6i + 5))\\
                %& = 2(\gcd((2i^2+ 6i + 5) + 2i + 3, 2i^2+ 6i + 5))\\
                 & = 2(\gcd(2i + 3, 2i^2+ 4i + 2 + (2i + 3) ))\\
                 & = 2(\gcd(2i + 3, 2(i + 1)^2))= 2(\gcd(2i + 3, i + 1))=2  \\
               % & = 2 (\gcd(i + 2 + i + 1, i + 1) )\\
               % & = 2(\gcd(i + 1 + 1, i + 1))=2
                %& = 2(\gcd( 1, i + 1 ))\\
                % = 2.
\end{align*}
Also,
\begin{align*}
\gcd(f + 2, 2m) &= 2(\gcd(2i^2+ 8i + 9, (2i^2+ 8i + 9) + 2i^2 + 6i + 4))\\
                % & = 2(\gcd((2i^2+ 6i + 4) + 2i + 5, 2i^2+ 6i + 4))\\
                % & = 2(\gcd(2i + 5, 2i^2+ 6i + 4 ))\\
                 & = 2(\gcd(2i + 5, 2(i + 2)(i + 1)))\\
                 & = 2(\gcd(2i + 5, (i + 2)(i + 1))).
\end{align*}
Now $\gcd(2i + 5, i + 2) = \gcd(2(i + 2) + 1, i + 2)=1.$
Whereas
\begin{align*}
  \gcd(2i + 5, i + 1) & = \gcd(2(i + 1) + 3, i + 1) )= \gcd(3, i + 1)\\
                            & \neq 1 \;\;\mbox{when} \;\; i + 1 \equiv 0\; \rm{mod}\; 3 \; \mbox{or equivalently}\;  i \equiv 2\; \rm{mod}\; 3.
\end{align*}

Thus taking $i \not\equiv 2 \; \rm{mod}\; 3$ we can construct a DCA$(4, 2m+1;2m)$ as per the Theorem \ref{QD(2m, 3)}.
\end{proof}

\subsection{Construction of difference covering arrays DCA$(4, 4m+1;4m)$}

In this subsection we give a general construction for a difference covering array DCA$(4, 4m+1;4m)$. It will be shown that for all non-negative integers $k$, such that $3 \!\nmid \!(2k + 1)$, this construction gives an infinite family of DCA$(4, 16k+9;16k+8)$. The proof that such a difference covering array exists uses the results presented in the following lemma. Note that in this section unless otherwise stated  all arithmetic is modulo $4m$.

\begin{lemma}\label{lem:gcd2}
Let $f$ and $m$ be natural numbers such that $m\equiv 2\mod 4$,
${\rm gcd}(f,$ \\$4m)=2$,
${\rm gcd}(f-1, 4m)=1$,
and
$f^2+f-2\equiv 2m\mod 4m$.
Then
\begin{align}
{\rm gcd}(2m+2-f, 4m)&=4,\label{gcd 2m+2-f}\\
{\rm gcd}(2m-f+1, 4m)&=1,\label{gcd 2m-f+1}\\
{\rm gcd}(2m-2f+2, 4m)&=2\label{gcd 2m-2f+2},\\
  mf&\equiv 2m\mod 4m.\label{gcd mf}
\end{align}
\end{lemma}
\begin{proof}

Eq \ref{gcd 2m+2-f}: \quad
Rewriting $2m+2-f = f^2+f-2+2-f = f^2 \mod 4m$ and assuming ${\rm gcd}(f, 4m)=2$ gives ${\rm gcd}(f^2, 4m)=4$.

Eq \ref{gcd 2m-f+1}: \quad Since $2m-f+1$ is odd,  the ${\rm gcd}(2m-f+1, 4m)$ is odd. Assume there exists an odd $x$ such that $x|4m$ and $x|(2m-f+1)$,  then $x|m$ and so $x|(f-1)$. But the ${\rm gcd}(f-1, 4m)=1$, so $x=1$.

Eq \ref{gcd 2m-2f+2}: \quad Assume that there exists $x$ such that $x|(m-f+1)$ and $x|2m$. Since $m-f+1$ is odd, $x$ is odd and so $x|m$. Consequently $x|(f-1)$ and $x|4m$, implying $x=1$.

Eq \ref{gcd mf}: \quad It follows that $  mf\equiv m(2m-f^2+2)= 2m^2-mf^2+2m\equiv 2m\mod 4m.$
\end{proof}

\begin{theorem} \label{QD(4m, 3)}
Let $f$ be a natural number and $m=4k+2$,  where $k$ is a non-negative integer, such that
${\rm gcd}(f, 4m)=2$,
${\rm gcd}(f-1, 4m)=1$,
and
$f^2+f-2\equiv 2m\mod 4m$.
Then a  cyclic DCA$(4, 4m+1;4m)$ satisfying P1 and P2 exists.
\end{theorem}
\begin{proof}

The proof is by construction with the values of DCA$(4, 4m+1;4m)$ as given in Figure \ref{values-4m2}, with the 3 columns of $Q=[q(i,j)]$ given by $q(a,0)=a$, $q(a, 1)=b(a)$, $q(a, 2)=c(a)$. We will show that $Q$ has the required properties.

\begin{figure}
\begin{center}
{\scriptsize
\begin{tabular}{|c|c|c|c|c|}
\hline
Intervals& $I_1$&$I_2$&$I_3$&$I_4$\\
\hline
 $q(a,0)=a$& $[0,m -1]$&$[m , 2m-1]$&$[2m, 3m -1]$&$[3m , 4m-1]$\\
\hline
$q(a, 1)=b(a)$  & $(a+1)f-1 $&$(a+1)f-1  $&$af$&$af $\\
\hline
$q(a, 2)=c(a)$  &$(a+1)(2m-f+2) $&$a(2m-f+2)$&$(a+1)(2m-f+2)$&$a(2m-f+2)$\\
  &$-1 $&$-m $&$+m  -1$&\\
\hline\hline
$b(a)-a$&  $(a+1)(f-1) $&$(a+1)(f-1)  $&$a(f-1)      $&$a(f-1)          $\\
\hline
$c(a)-a$&$(a+1)(2m-f+1) $&$a(2m-f+1) $&$(a+1)(2m-f+1) $&$a(2m-f+1) $\\
&&$-m $&$+m $&\\
\hline
$c(a)-b(a)$&$(a+1)(2m-2f+2)$&$a(2m-2f+2)$&$a(2m-2f+2)$&$a(2m-2f+2) $\\
&&$-f-m +1$&$+3m -f+1$&\\
\hline
\end{tabular}
}
\end{center}
\caption{Entries are elements of ${\mathbb Z}_{4m}$, where $q(a,0)=a,$  $q(a, 1)=b(a)$ and $q(a, 2)=c(a)$ in the array $Q=[q(i,j)]$.}\label{values-4m2}
\end{figure}

For  $b(a)$, since $f$ is even,  Lemma \ref{basic} implies that
  \begin{align*}
(a+1)f-1&\equiv 1\mod 2 \Longrightarrow \{b(a)\mid a\in I_1\cup I_2\}=\{2g+1\mid 0\leq g \leq 2m-1\},\\
af&\equiv 0\mod 2 \Longrightarrow  \{b(a)\mid a\in I_3\cup I_4\}=\{2g\mid 0\leq g \leq 2m-1\},
\end{align*}
and $\{b(a)\mid 0\leq a\leq 4m-1\} ={\mathbb Z}_{4m}$.

For $c(a)$, since ${\rm gcd}(2m-f+2, 4m)=4$, Lemma \ref{basic} implies that
$$\begin{array}{rl}
(a+1)(2m-f+2)+m -1\equiv 1 \mod 4 \Longrightarrow &\{c(a)\mid a\in I_3\}=\{4g+1\mid 0\leq \\
& g\leq m -1\},\\
a(2m-f+2)\equiv 0 \mod 4  \Longrightarrow &\{c(a)\mid a\in I_4\}=\{4g\mid 0\leq g\leq\\ & m -1\},\\
(a+1)(2m-f+2)-1\equiv 3 \mod 4  \Longrightarrow &\{c(a)\mid a\in I_1\}=\{4g+3\mid 0\leq \\ &g\leq m -1\},\\
a(2m-f+2)-m \equiv 2 \mod 4 \Longrightarrow &\{c(a)\mid a\in I_2\}=\{4g+2\mid 0\leq \\&g\leq m -1\}.
\end{array}$$
Thus the set of values $\{c(a)\mid a\in {\mathbb Z}_{4m}\}={\mathbb Z}_{4m}$.

For $b(a)-a$, $c(a)-a$ and $c(a)-b(a)$ we are required to show that for $a\in{\mathbb Z}_{4m}$ the  differences cover the multiset $\{1, 2,\dots, 2m-1, 2m, 2m, 2m+1,\dots, 4m-1\}=({\mathbb Z}_{4m}\setminus \{0\})\cup\{2m\}$.

For $b(a)-a$, the ${\rm gcd}(f-1, 4m)=1$, and
 \begin{align*}
b(2m)-2m & =2m(f-1)= 2m,\\
b(2m-1)-(2m-1) & =(2m-1+1)f-1-(2m-1)=(2m)(f-1)=2m,\\
 b(4m-1)-(4m-1)&= -(f-1),\\
 b(0)-0&=f-1.
 %b(0)-0-(b(4m-1)-(4m-1))&\equiv 2(f-1)
\end{align*}
So  using a ``jump'' of  $f-1$ and ordering the subintervals as $I_3,I_4,I_1,I_2$ we obtain the difference $2m$ twice and   the difference 0 is omitted between $a=4m-1$ and $a=0$ implying that  $\{b(a)-a\mid0\leq a\leq 4m-1\}=({\mathbb Z}_{4m}\setminus \{0\})\cup\{2m\}$.

For $c(a)-a$, the ${\rm gcd}(2m-f-1, 4m)=1$, and
$$\begin{array}{rcl}
c(m )-m  & = & m (2m-f+1)-m =2m,\\
c(3m -1)-(3m -1) & = &(3m -1+1)(2m-f+1)+m\\& =&2m,\\
c(3m )-3m -(c(2m-1)-(2m-1))&=&(m +1)(2m-f+1)+m \\ &=& 2m-f+1,\\
c(0)-0&=&2m-f+1,\\
c(4m-1)-(4m-1)&=&(4m-1)(2m+f-1)\\ &=&-(2m-f+1),\\
%c(0)-0-(c(4m-1)-(4m-1))&=&2(2m-f+1)\\
c(2m)-2m-(c(m -1)-(m -1))&=&(m +1)(2m-f+1)+m \\ &=& 2m-f+1.
\end{array}$$
So using a ``jump'' of  $2m-f+1$ and ordering the subintervals as $I_2,I_4,I_1,I_3$ we obtain the difference $2m$ twice and   the difference 0 is omitted between $a=4m-1$ and $a=0$, implying   $\{c(a)-a\mid0\leq a\leq 4m-1\}=({\mathbb Z}_{4m}\setminus \{0\})\cup\{2m\}$.

For $c(a)-b(a)$,   and
 \begin{align*}
c(3m )-b(3m ) & =m(2m-2f+2)=2m,\\
c(m-1)-b(m-1) & = (m-1+1)(2m-2+2)=2m,\\
c(4m-1)-b(4m-1)&=
%(4m-1)(2m-2f+2)=
-(2m-2f+2),\\
c(0)-b(0)&=2m-2f+2,\\
c(2m)-b(2m)-(c(2m-1)-b(2m-1))&=%2m(2m-2f+2)+3m -f+1\\
%&&-(2m-1)(2m-2f+2)+f+m -1\\
2m-2f+2.
\end{align*}
Then since
$$\begin{array}{rcl}
2m-2f-2\equiv 0\mod 2 \Longrightarrow \{c(a)-b(a)\mid a\in I_4\cup I_1\}&=&\{2g\mid 1\leq g \leq\\&& 2m-1\}\cup\{2m\}, \\
-f+1\equiv 1\mod 2 \Longrightarrow \{c(a)-b(a)\mid a\in I_2\cup I_3\}&=&\{2g+1\mid 0\leq g \leq \\&&2m-1\},
\end{array}$$
implying $\{c(a)-b(a)\mid0\leq a\leq 4m-1\}=({\mathbb Z}_{4m}\setminus \{0\})\cup\{2m\}$.
\end{proof}

\begin{corollary}\label{cor-4m+1}
For $k\geq 0$ such  that $k \not\equiv 1 \mod 3$ a cyclic DCA$(4, 4m+1;4m)$ satisfying P1 and P2
can be constructed as described in Theorem \ref{QD(4m, 3)}.
\end{corollary}
\begin{proof}
 Given $m=4k+2$, take $f=2m-2$. Then $f=8k+2$, and
\begin{align*}
 \gcd(f, 4m) & = 2(\gcd(4k + 1, 8k + 4)) = 2(\gcd(4k + 1, 2(4k + 1) + 2))= 2.
\end{align*}
In addition
\begin{align*}
\gcd(f-1, 4m) & = \gcd(2m - 3, 4m) = \gcd(8k + 1, 16k + 8)\\
                   %& = \gcd(8k + 1, 8(2k + 1)\\
                   & = \gcd(8k + 1, 2k + 1) \;\;\rm{since} \;\;2 \!\nmid\! (8k + 1)\\
                   & = \gcd(6k, 2k + 1)
                 % & = \gcd(3(2k), 2k + 1)\\
                   = 1, \;\rm{if}\; 3\!\nmid\! (2k + 1).
\end{align*}
 Also
\begin{align*}
f^2+f-2 & = 64k^2+32k+4+8k+2-2 = 4k(16k+8)+8k+4 \equiv 2m\; \rm{mod}\; 4m.
\end{align*}

Hence, $f=2m-2$ satisfies the assumptions of Theorem \ref{QD(4m, 3)} and we can construct a  DCA$(4, 16k + 9;16k+8)$ for $k$ such  that $3 \!\nmid \!(2k + 1)$ as described in this theorem.
\end{proof}

\subsection{Construction of difference covering arrays  DCA$(4, 2m+1;2m)$, where $m=3\mu + 2$}

In this subsection we give a general construction for a difference covering array DCA$(4, 2m+1;2m)$, where $m=3\mu + 2$. The proof that such a difference covering array exists uses the result presented in the following lemma. Note that in this section unless otherwise stated all arithmetic is modulo $12k+10$, $k\geq 0$.

\begin{theorem}\label{thm-6mu+4}
Let  $\mu$ be an odd positive integer.
Then there exists a cyclic DCA$(4, 6\mu+5;6\mu+4)$ satisfying P1 and P2.
\end{theorem}

\begin{proof}  Since $\mu\geq 1$, $n=6\mu+4\geq 10$, since $3$ is prime, ${\rm gcd}(3, 6\mu+4)=1$.
%From Lemma \ref{basic0},  ${\rm gcd}(6, 6\mu+4)=2$, hence ${\rm gcd}(3, 6\mu+4)\neq 3$, thus ${\rm gcd}(3, 6\mu+4)=1$.
%${\rm gcd}(3, 6\mu+4)={\rm gcd}(3\mu, 3\mu+2)={\rm gcd}(3\mu+1, 3\mu+2)={\rm gcd}(3\mu+3, 3\mu+2)={\rm gcd}(3\mu+4, 3\mu+2)=1$.
Hence ${\rm gcd}(3\mu+4, 6\mu+4)={\rm gcd}(3\mu, 6\mu+4)=1$.

Let $\mu=2k+1$, $k\geq 0$, then
\begin{align*}
\mu(3\mu+2)&=(2k+1)(6k+5)=3\mu+2,\mbox{ and}\\
(3\mu+2)^2&=3\mu+2.
\end{align*}
That is, $(n/2)^2\equiv n/2 \mod n$.

The proof is by construction with the values for DCA$(4, 6\mu+5;6\mu+4)$ as given  in Figure \ref{values-3}, with the three columns of  $Q=[q(i,j)]$  given by $q(\alpha,0)=a(\alpha)$, $q(\alpha, 1)=b(\alpha)$  and $q(\alpha, 2)=c(\alpha)$.

\begin{figure}
\begin{center}
\begin{tabular}{|c|c|c|c|}
\hline
Intervals for& $I_1$&$I_2$&$I_3$\\
 $\alpha$&$[0,\mu-1]$&$[\mu, 2\mu]$&$[2\mu+1, 3\mu+1]$\\
\hline
$q(\alpha,0)=a(\alpha)$&$3\alpha+3\mu+4$&$3\alpha+2$&$3\alpha  +3\mu+4$\\
\hline
$q(\alpha, 1)=b(\alpha)$&$3\alpha(\mu+1)+2\mu+2$&$3\alpha(\mu+1)+2\mu+2$ &$3\alpha(\mu+1)+2\mu+2$ \\
\hline
$q(\alpha, 2)=c(\alpha)$&$\alpha(3\mu+4) +5\mu+4$&$\alpha(3\mu+4) +5\mu+4$ &$\alpha(3\mu+4) +5\mu+4$\\
\hline\hline
$b(\alpha)-a(\alpha)$&$3\alpha\mu-\mu-2$&$3\alpha\mu+2\mu$&$3\alpha\mu-\mu-2$\\
\hline
$c(\alpha)-a(\alpha)$&$\alpha(3\mu+1)+2\mu$&$\alpha(3\mu+1) +5\mu+2$&$\alpha(3\mu+1) +2\mu$ \\
\hline
$c(\alpha)-b(\alpha)$&$\alpha+3\mu+2$&$\alpha+3\mu+2$&$\alpha+3\mu+2$\\
\hline
\hline
\hline
Intervals for& $I_4$&$I_5$&$I_6$\\
 $\alpha$&$[3\mu+2, 4\mu+2]$&$[4\mu+3, 5\mu+2]$&$[5\mu+3, 6\mu+3]$\\
\hline
$q(\alpha,0)=a(\alpha)$&$3\alpha+3\mu+3$&$3\alpha+1$&$3\alpha+3\mu+3$\\
\hline
$q(\alpha, 1)=b(\alpha)$&$3\alpha(\mu+1)+2\mu+1$&$3\alpha(\mu+1)+2\mu+1$ & $3\alpha(\mu+1)+2\mu+1$ \\
\hline
$q(\alpha, 2)=c(\alpha)$&$\alpha(3\mu+4) +5\mu+4$&$\alpha(3\mu+4) +5\mu+4$ &$\alpha(3\mu+4) +5\mu+4$ \\
\hline\hline
$b(\alpha)-a(\alpha)$&$3\alpha\mu-\mu-2$&$3\alpha\mu+2\mu$&$3\alpha\mu-\mu-2$\\
\hline
$c(\alpha)-a(\alpha)$&$\alpha(3\mu+1)+2\mu+1$&$\alpha(3\mu+1)+5\mu+3$&$\alpha(3\mu+1)+2\mu+1$
  \\
\hline
$c(\alpha)-b(\alpha)$&$\alpha+3\mu+3$&$\alpha+3\mu+3$&$\alpha+3\mu+3$\\
\hline
\end{tabular}
\end{center}
\caption{Entries are elements of ${\mathbb Z_{6\mu+4}}$, where $q(\alpha,0)=a(\alpha),$  $q(\alpha, 1)=b(\alpha)$ and $q(\alpha, 2)=c(\alpha)$ in the array $Q=[q(i,j)]$.}\label{values-3}
\end{figure}

Since ${\rm gcd}(3,n)=1$,   $\{3\alpha\mid 0\leq \alpha \leq n-1\}={\mathbb Z}_n$, by Lemma \ref{basic}.
Further $a(3\mu+2)=3(3\mu+2)+3\mu+3=1$ and $  a(2\mu)=6\mu+2$ and there is a ``jump'' of $3$ between  $a(4\mu+2)$ and $a(0)$;
$a(\mu-1)$ and $a(5\mu+3)$;  $a(6\mu+3)$ and $a(2\mu+1)$; $a(3\mu+1)$ and $a(4\mu+3)$;
 $a(5\mu+2)$ and $  a(\mu)$, respectively.
%\begin{align*}
%a(0)-a(4\mu+2)&=3\mu +4+3(4\mu+2)+3\mu+3\equiv 3\mod 6\mu+4\\
%a(5\mu+3)-a(\mu-1)&=&3(5\mu+3)+3\mu+3-(3(\mu-1)+3\mu+4)\equiv 3\mod 6\mu+4\\
% a(2\mu+1)-a(6\mu+3)&=&3(2\mu+1)+3\mu+4-(3(6\mu+3)+3\mu+3)\equiv 3\mod 6\mu+4\\
% a(4\mu+3)-a(3\mu+1)&=&3(4\mu+3)+1-(3(3\mu+1)+3\mu+4)\equiv 3\mod 6\mu+4\\
%  a(\mu)-a(5\mu+2)&=&3\mu+2-(3(5\mu+2)+1)\equiv 3\mod 6\mu+4\\
%  a(2\mu)&=&3(2\mu)+2=6\mu+2,
%  \end{eqnarray*}
Thus  reordering the subintervals as $I_4,I_1,I_6,I_3,I_5,I_2$ gives $\{a(\alpha)\mid 0\leq \alpha\leq  n-1\}={\mathbb Z}_n$.

For $b(\alpha)$,
$$\begin{array}{rl}
3\alpha(\mu+1)+2\mu+2\equiv 0 \mod 2 \Longrightarrow &
\{b(\alpha)\mid  \alpha \in I_1 \cup I_2 \cup I_3\}\\&=\{2g\mid 0\leq g\leq 3\mu+2\}\\
3\alpha(\mu+1)+2\mu+1\equiv 1 \mod 2 \Longrightarrow &
\{b(\alpha)\mid  \alpha \in I_4 \cup I_5 \cup I_6\}\\&=\{2g+1\mid 0\leq g\leq 3\mu+1\}
\end{array}$$

For $c(\alpha)$, since ${\rm gcd}(3\mu+4, 6\mu+4)=1$  Lemma \ref{basic} implies that
$\{c(\alpha)\mid 0\leq \alpha\leq 6\mu+3\}={\mathbb Z_{6\mu+4}}$.

For $b(\alpha)-a(\alpha)$,
since ${\rm gcd}(3\mu, 6\mu+4)=1$,
 \begin{align*}
b(\mu)-a(\mu) & =3\mu+2,\\
b(4\mu+2)-a(4\mu+2) & = 3\mu+2,\\
b(5\mu+3)-a(5\mu+3)-(b(2\mu)-a(2\mu))&=%3(5\mu+3)\mu-\mu-2-(3(2\mu)\mu+2\mu)\\
%&=&9\mu^2+6\mu-2=3\mu(3\mu+2)-2+3\mu^2\equiv
3\mu\\
%b(\mu-1)-a(\mu-1)&=% 3(\mu-1)\mu-\mu-2\equiv
%-3\mu\\
b(4\mu +3)-a(4\mu+3)-(b(\mu-1)-a(\mu-1))&=%& 3(4\mu+3)\mu+2\mu-(3(\mu-1)\mu-\mu-2)\\
 6\mu \\
b(2\mu+1)-a(2\mu+1)-(b(5\mu +2)-a(5\mu+2))&=%& 3(2\mu+1)\mu-\mu-2-(3(5\mu+2)\mu+2\mu)\\
%&=& -9\mu^2-6\mu-2\equiv
3\mu\\
b(0) - a(0) -(b(6\mu + 3) - a(6\mu + 3) & = 3\mu.
\end{align*}
So using a ``jump'' of  $3\mu$ and ordering the subintervals as $I_2,I_6,I_1,I_5,I_3, I_4,$ we obtain the difference $3\mu+2$ twice and  since there is $6\mu$ between $b(\alpha)-a(\alpha)$ for $\alpha=\mu-1$ and $\alpha=4\mu+3$ the difference 0 is omitted  implying that  $\{b(\alpha)-a(\alpha)\mid 0\leq \alpha \leq 6\mu+3\}=({\mathbb Z}_n\setminus \{0\})\cup\{n/2\}$.

For $c(\alpha)-b(\alpha)$, since ${\rm gcd}(3\mu+1, 6\mu+4)=2$, we have \\
 and
  \begin{align*}
c(5\mu+3)-a(5\mu+3) & = 3\mu+2,\\
c(2\mu)-a(2\mu) & = 3\mu+2,\\
c(3\mu+2)-a(3\mu+2)-(c(6\mu+3)-a(6\mu+3))&=%-(3\mu+2)(3\mu+1)-2\mu-1-(-(6\mu+3)(3\mu+1)-2\mu-1)\\
 -(3\mu+3)\\
c(\mu)-a(\mu)-(c(4\mu+2)-a(4\mu+2))&=%-\mu(3\mu+1)-5\mu-2-(-(4\mu+2)(3\mu+1)-2\mu-1)\\
-(3\mu+3)\\
c(2\mu+1)-a(2\mu+1)&=3\mu+1\\
c(5\mu+2)-c(5\mu+2)&=3\mu+3\\
c(0)-a(0)-(c(3\mu+1)-a(3\mu+1))&%=&-2\mu-(-(3\mu+1)(3\mu+1)-2\mu)\\
=-(3\mu+3)\\
c(\mu-1)-a(\mu-1)-(c(4\mu+3)-a(4\mu+3))&=%&-(4\mu+3)(3\mu+1)-5\mu-3)-(-(mu-1)(3\mu+1)-2\mu)\\
 -(3\mu+3).
\end{align*}
Reordering the intervals as $I_6,I_4,I_2$ and $I_3,I_1,I_5$ and  using a regular ``jump'' of $-(3\mu+3)$, with the jump of $6\mu+2$ between $\alpha=2\mu+1$ and $\alpha=5\mu+2$, being the exception, we have the difference $3\mu+2$ twice and the difference $0$ omitted, thus $\{c(\alpha)-a(\alpha)\mid 0\leq \alpha\leq 6\mu+3\}=({\mathbb Z_{6\mu+4}}\setminus\{0\})\cup \{3\mu+2\}$ with repetition retained.

For  $c(\alpha)-b(\alpha)$, we note that
\begin{align*}
c(0)-b(0)&= 3\mu+2,\\
c(3\mu+1)-b(3\mu +1)&=- 1,
\end{align*}
and so the values of $c(\alpha)-a(\alpha)$ on the subinterval $I_1\cup I_2\cup I_3$ cover the set $\{3\mu+2,\dots,-1\}$. Also
\begin{align*}
c(3\mu+2)-b(3\mu+2)&=1,\\
c(6\mu+3)-b(6\mu +3)&= 3\mu+2,
\end{align*}
and so the values of $b(\alpha)-c(\alpha)$ on the subinterval $I_4\cup I_5\cup I_6$ cover the set $\{3\mu+2,\dots, 6\mu-1\}$. Consequently  $\{b(\alpha)-c(\alpha)\mid 0\leq \alpha\leq 6\mu+3\}=([6\mu+4]\setminus\{0\})\cup \{3\mu+2\}$ with repetition retained.
\end{proof}

\subsection{Infinite families}
The construction of Theorem \ref{thm-6mu+4} constructs sets of three MNOLS of orders $10, 22, 34, 46 \mod 48$.  The construction of Corollary \ref{cor-4m+1} constructs sets of three MNOLS of orders $8, 40\mod 48$.
Combined with the constructions of \cite{DDK}, there is a construction of three MNOLS for $8, 10, 14, 22, 34,$
$38, 40, 46\mod 48$.  There are  infinite families constructed from Corollary \ref{cor-2m} and from results of Li and van Rees \cite{LvR}, but these cannot be described mod 48. It is an open question as to why 48 features in many of the constructions.


\begin{thebibliography}{99}
\bibitem {KAG} Abdel-Ghaffar, K.A.S,
{\em On the number of mutually orthogonal partial Latin squares},
Ars Combinatoria {\bf 42} (1996), 259--286.

\bibitem{BB}  Bate, S.T. and Boxall, J.,
{\em The construction of multi-factor crossover designs in animal husbandry studies},
Pharmaceutical Statistics{\bf 7} (2008), 179--194.

\bibitem{BHS} Bussemaker, F. C., Haemers, W. H. and Spence, E.,
{\em The search for pseudo orthogonal Latin squares of order six},
Designs, Codes and Cryptography{\bf 21} (2000), 77--82.

\bibitem{CDFP1} Cohen D.M., Dalal, S.R., Fredman, M.L., and Patton, G.C.,
{\em The AETG system: an approach to testing based on combinatorial designs},
IEEE Trans Software Eng {\bf 23} (1997), 437--444.

\bibitem{CDPP2} Cohen D.M., Dalal, S.R., Parelius, J., and Patton, G.C.,
{\em The combinatorial design approach to automatic test generation},
IEEE Trans Software {\bf 13} (1996), 83--88.

\bibitem{colbourn}  Colbourn, C.J. and  Dinitz, J.H.,  (Eds.),
Handbook of combinatorial designs, Second Edition. Chapman \& Hall/CRC, Boca Raton, FL, 2006.
press, 2010.

\bibitem{DDK} Demirkale, F.,  Donovan, D. and Khodkar, A.,
{\em Direct constructions for general families of cyclic mutually nearly orthogonal Latin squares},
Journal of Combinatorial Designs, 2014, doi: 10.1002/jcd.21394.

\bibitem{Ge} Ge, G., 
{\em On $(g, 4;1)$-difference matrices},
Discrete Mathematics {\bf 301} (2005), 164--174.

\bibitem{HSS}  Hedayat, A.S., Sloane, N.J.A. and  Stufken, John,
Orthogonal Arrays: Theory and Applications. Springer, New York, 1999.

\bibitem{Korner} Korner, J. and  Lucertini, M.,
{\em Compressing inconsistent data},
IEEE Trans. Inform. Theory {\bf 40} (1994), 706-–715

\bibitem {LvR}  Li, P.C.  and  van Rees, G.H.J.,  
{\em Nearly orthogonal Latin squares},
Journal of Combinatorial Mathematics and Combinatorial Computing{\bf 62} (2007), 13--24.

\bibitem{PR} Pasles, E.B. and Raghavarao, D.,
{\em Mutually nearly orthogonal Latin squares of order 6},
Utilitas Mathematica {\bf 65} (2004), 65--72.

\bibitem{RSS}  Raghavarao, D.,  Shrikhande, S.S. and  Shirkhande, M.S.,
{\em Incidence matrices and inequalities for combinatorial designs},
Journal of Combinatorial Designs{\bf 10} (2002), 17--26.

\bibitem{SMM}  Stevens, B., Moura, L. and Mendelsohn, E.,
{\em Lower bounds for transversal covers},
Designs, Codes and Cryptography{\bf 15} (1998), 279--299.

\bibitem{Stinson} Stinson, D.R.,
{\em Combinatorial characterizations of authentication codes},
Designs, Codes and Cryptography{\bf 2} (1992), 175--187

 \bibitem{todorov}  Todorov,  D.T., 
 {\em Four Mutually  Orthogonal Latin Squares of Order 14},
Journal of Combinatorial Design{\bf 20} (2012), 363--367.

\bibitem{priv-van-rees}  van Rees, G.H.J., Private Communication (2014).

\bibitem{W} Williams, E.J.,
{\em Experimental designs balanced for the estimation of residual effects of treatments},
Australian Journal of Scientific Research \textbf{2} (1949), 149--168.

\bibitem{Yin1}  Yin, J.,
{\em Construction of difference covering arrays},
Journal of Combinatorial Theory, Series A{\bf 104} (2003), 327--339.

\bibitem{Yin2}  Yin, J.,  
{\em Cyclic difference packing and covering arrays},
Designs, Codes and Cryptography{\bf 37} (2005), 281--292.

\end{thebibliography}
\end{document}